\documentclass[11pt,reqno]{amsart}
\usepackage{indentfirst, amssymb,amsmath,amsthm, mathrsfs,cite,graphicx,float}
\usepackage{tikz}
\usepackage{float}
\usetikzlibrary{arrows.meta} 
\newtheorem*{theoA}{Theorem A}
\newtheorem*{theoB}{Theorem B}
\newtheorem*{theoC}{Theorem C}
\newtheorem*{theoD}{Theorem D}
\newtheorem*{theoE}{Theorem E}
\newtheorem*{theoF}{Theorem F}

\newtheorem{theo}{Theorem}[section]
\newtheorem{lem}{Lemma}[section]
\newtheorem{prop}{Proposition}[section]

\newtheorem{cor}{Corollary}[section]

\newtheorem{rem}{Remark}[section]
\newtheorem{exam}{Example}[section]

\newtheorem{open problem}{Open problem}[section]

\newcommand{\ol}{\overline}
\newcommand{\be}{\begin{equation}}
\newcommand{\ee}{\end{equation}}
\newcommand{\bs}{\begin{small}}
\newcommand{\es}{\end{small}}
\newcommand{\beas}{\begin{eqnarray*}}
\newcommand{\eeas}{\end{eqnarray*}}
\newcommand{\bea}{\begin{eqnarray}}
\newcommand{\eea}{\end{eqnarray}}
\renewcommand{\epsilon}{\varepsilon}
\numberwithin{equation}{section}
\begin{document}
\title[Landau-type theorems]{Landau-type theorems for $K$-quasiregular harmonic mappings}
\author[V. Allu and R. Biswas]{Vasudevarao Allu and Raju Biswas}
\date{}
\address{Vasudevarao Allu, Department of Mathematics, School of Basic Sciences, Indian Institute of Technology Bhubaneswar, Bhubaneswar-752050, Odisha, India.}
\email{avrao@iitbbs.ac.in}
\address{Raju Biswas, Department of Mathematics, Raiganj University, Raiganj, West Bengal-733134, India.}
\email{rajubiswasjanu02@gmail.com}
\maketitle
\let\thefootnote\relax
\footnotetext{2020 Mathematics Subject Classification: 31A30, 30C99.}
\footnotetext{Key words and phrases: Landau-type theorem, Bloch theorem, $K$-quasiregular harmonic mappings, univalent functions, Bloch space.}
\begin{abstract} 
In this paper, our aim is to establish several Landau-type theorems for $K$-quasiregular harmonic mappings $f=h+\overline{g}$ in the unit disk $\Bbb{D} = \{z\in\Bbb{C}: |z|<1\}$. 
Under various boundedness assumptions on the analytic part $h$, we obtain explicit univalence radii and corresponding schlicht disk radii that improve upon existing estimates in the literature. We also establish new Landau-type theorems under novel  
hypotheses. We provide examples to illustrate our results, and comprehensive numerical tables present quantitative values of the radii for various parameter choices, demonstrating the effectiveness of our results.
\end{abstract}
\section{Introduction and Preliminaries}
Let $\Bbb{D}_r(a):=\{z\in\Bbb{C} : |z-a|<r\}$ denote the disk of radius $r>0$ with centre at the point $a\in\Bbb{C}$. In particular, $\Bbb{D}_r:=\Bbb{D}_r(0)$ and 
$\Bbb{D}:=\Bbb{D}_1(0)$ be the unit disk. Let $A(\Bbb{D})$ denote the class of all analytic functions in $\Bbb{D}$.
A complex-valued function $f$ is said to be harmonic in the unit disk $\Bbb{D}$ if $f$ is twice continuously differentiable and satisfies the Laplace equation $\Delta f:=4f_{z\overline{z}}=0$ in $\Bbb{D}$. 
A harmonic mapping $f$ in the unit disk $\Bbb{D}$ can be represented as $f = h + \overline{g}$ with $f(0)=h(0)$, where $h$ and $g$ are holomorphic functions in $\Bbb{D}$ (see \cite{D2004}). Throughout this paper, we use the canonical decomposition $f=h+\overline{g}$ normalized by $f(0)=h(0)$, and hence $f(0)=0$ implies
$h(0)=g(0)=0$.
For a continuously differentiable function $f$, the maximum and minimum distortions (see \cite{CGH2000}) are defined as follows: 
\beas \Lambda_f(z) &=& \max_{0\leq \theta \leq 2\pi} \left|f_z+e^{-2i\theta} f_{\ol{z}}\right|= |f_z(z)| + |f_{\bar{z}}(z)|\quad \text{and}\\[2mm] 
\lambda_f(z)& =& \min_{0\leq \theta \leq 2\pi} \left|f_z+e^{-2i\theta} f_{\ol{z}}\right|=\left| |f_z(z)| - |f_{\bar{z}}(z)| \right|,\eeas
where $f_z = (1/2)(f_x - i f_y)$ and $f_{\bar{z}} = (1/2)(f_x + i f_y)$. 
The Jacobian is given by $J_f(z) =|f_z(z)|^2-|f_{\ol{z}}(z)|^2= |h'(z)|^2 - |g'(z)|^2$, and note that $|J_f(z)|=\Lambda_f(z) \lambda_f(z)$. The inverse function theorem and a 
result of Lewy \cite{L1936} demonstrates that a harmonic function $f$ is locally univalent in $\Omega$ if, and only if, its Jacobian $J_f(z)\not=0$ in $\Omega$.
Furthermore, a harmonic 
mapping is said to be sense-preserving if $J_f > 0$. A sense-preserving harmonic mapping is called $K$-quasiregular ($K \geq 1$) in $\Bbb{D}$
if $\Lambda_f(z) \leq K \lambda_f(z)$ for all $z \in \Bbb{D}$.
 We say that a disk $\Bbb{D}_\rho(a)$ is a schlicht disk of $f$ if there is a subdomain $\Omega\subset \Bbb{D}$ such that $f$ is univalent in $\Omega$ and $f(\Omega)=\Bbb{D}_\rho(a)$.
\subsection{The classical Bloch Theorem and Landau Theorem} 
In 1925, Bloch \cite{B1925} discovered a covering theorem for non-univalent analytic functions in the unit disk $\Bbb{D}$ requiring only the condition $f'(0)=1$. Thus, the classical Bloch theorem states that if $f$ is 
analytic in the unit disk $\Bbb{D}$ with $f'(0) = 1$, then the image $f(\Bbb{D})$
necessarily contains a schlicht disk of some positive radius. Let $b(f)$ be the supremum of all such positive radii. Then $$B=\inf\{b(f): f\in A(\Bbb{D}), f'(0)=1\}$$ is called the Bloch constant (see \cite{CGH2000}), whose exact 
value remains an open problem despite significant efforts (see \cite{CG1996, AG1937}). Shortly thereafter, Landau \cite{L1926} proved a result that explicitly connects the growth of an analytic function to its univalence and covering properties, now known as the classical Landau theorem.
It states that if $f$ is an analytic function in the unit disk $\Bbb{D}$ with $f(0)=0$, $f'(0)=1$, and $|f(z)|<M$ in $\Bbb{D}$ for some $M>1$, then $f$ is univalent in $\Bbb{D}_{r_0}$ with 
\beas r_0=\frac{1}{M+\sqrt{M^2-1}},\eeas
and $f(\Bbb{D}_{r_0})$ contains a disk $\Bbb{D}_{R_0}$ with $R_0=Mr_0^2$. This result is sharp, with the extremal function $f_0(z)=Mz(1-Mz)/(M-z)$. The Landau theorem is often called the Landau-Bloch theorem because it sharpens Bloch's existence result by giving an explicit radius in terms of a growth bound. Figure \ref{landau} illustrates the geometric content of the classical Landau theorem.
\begin{figure}[H]
\centering
\includegraphics[scale=0.75]{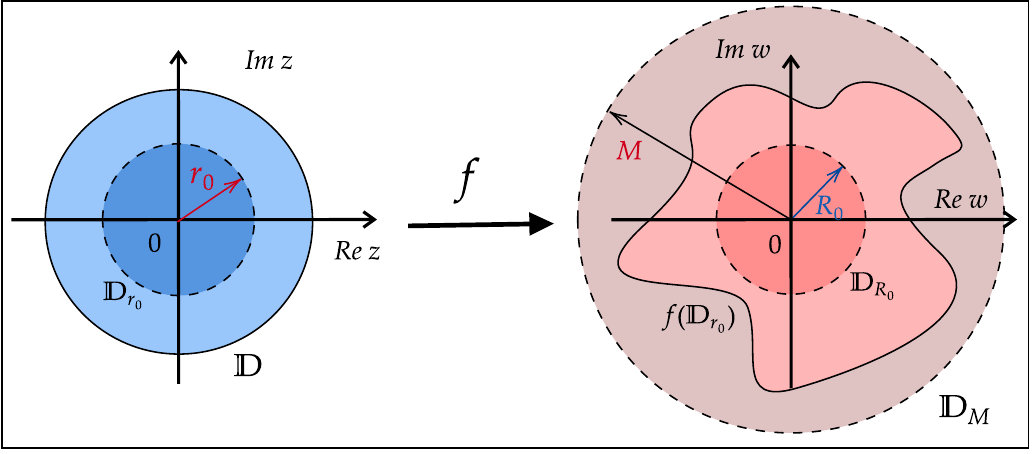}
\caption{Illustration of the classical Landau theorem}
\label{landau}
\end{figure}
An analytic function $f$ in $\Bbb{D}$ is said to be a Bloch function (see \cite{P1970, GK2003}) if 
\beas \|f\|_{\mathcal{B}}=\sup_{z\in\Bbb{D}} \left(1-|z|^2\right)|f'(z)|<\infty.\eeas 
In this case $\|f\|_{\mathcal{B}}$ is called the Bloch semi-norm of $f$. The space $\mathcal{B}$ of analytic Bloch functions in the unit disk $\Bbb{D}$ forms a Banach space
under the norm given by
\beas \|f\|=\|f\|_\mathcal{B}+|f(0)|.\eeas
In 2000, Chen {\it et al.} \cite{CGH2000} proved the following Landau-type theorems for harmonic mappings in the unit disk $\Bbb{D}$. 
\begin{theoA}\cite{CGH2000}
Let $f$ be a harmonic mapping of the unit disk $\Bbb{D}$ such that $f(0)=0$, $f_{\ol{z}}(0)=0$, $f_z(0)=1$, and $|f(z)|<M$ for $z\in\Bbb{D}$. Then, $f$ is univalent in a disk $\Bbb{D}_{\rho_0}$
with $\rho_0=\pi^2/(16mM)$ and $f(\Bbb{D}_{\rho_0})$ contains a schlicht disk $\Bbb{D}_{\sigma_0}$ with $\sigma_0=\rho_0/2=\pi^2/(32mM)$, where $m\approx 6.85$ is the minimum of the function $(3-r^2)/\left(r(1-r^2)\right)$ for $0<r<1$. 
\end{theoA}
\begin{theoB}\cite{CGH2000}
Let $f$ be a harmonic mapping of the unit disk $\Bbb{D}$ such that $f(0)=0$, $\lambda_f(0)=1$, and $\Lambda_f(z)\leq \Lambda$ for $z\in\Bbb{D}$. Then, $f$ is univalent in a disk $\Bbb{D}_{\rho_1}$
with $\rho_1=\pi/(4(1+\Lambda))$ and $f(\Bbb{D}_{\rho_1})$ contains a schlicht disk $\Bbb{D}_{\sigma_1}$ with $\sigma_1=\rho_1/2=\pi/(8(1+\Lambda))$.
\end{theoB}
It is important to note that \textrm{Theorems A} and \textrm{B} are not sharp. Better estimates were subsequently provided by Dorff and Nowak \cite{DN2004}. Subsequently, Grigoyan \cite{G2006} has proved the following results, which improves the estimates in \textrm{Theorems A} and \textrm{B}.  
\begin{theoC}\cite{G2006}
Let $f$ be a harmonic mapping of the unit disk $\Bbb{D}$ such that $f(0)=h(0)=0$, $J_f(0)=1$ and $|f(z)|<M$ for $z\in\Bbb{D}$. Then, $f$ is univalent on a disk $\Bbb{D}_{\rho_2}$
with $$\rho_2=1-\frac{2\sqrt{2}M}{\sqrt{\pi+8M^2}}$$ and $f(\Bbb{D}_{\rho_2})$ contains a schlicht disk $\Bbb{D}_{\sigma_2}$ with 
$$\sigma_2=\frac{\pi}{4M}+4M-4M\sqrt{1+\frac{\pi}{8M^2}}.$$ 
\end{theoC}
\begin{theoD}\cite{G2006}
Let $f$ be a harmonic mapping of the unit disk $\Bbb{D}$ such that $f(0)=0$, $\lambda_f(0)=1$, and $\Lambda_f(z)\leq \Lambda$ for $z\in\Bbb{D}$. Then, $f$ is univalent on a disk $\Bbb{D}_{\rho_3}$
with $\rho_3=1/(1+\Lambda)$ and $f(\Bbb{D}_{\rho_3})$ contains a schlicht disk $\Bbb{D}_{\sigma_3}$ with $\sigma_3=1-\Lambda \log\left(1+\frac{1}{\Lambda}\right)$.
\end{theoD}
In 2026, Wang and Zhong \cite{WZ2026} established the following Landau-type theorem for $K$-quasiregular harmonic mapping.
\begin{theoE}\cite{WZ2026} Suppose that $f(z)=h(z)+\ol{g(z)}$ is a harmonic mapping in the unit disk $\Bbb{D}$ with $|g'|\leq k|h'|$, $k\in[0, 1)$, where $h(z)=\sum_{n=1}^\infty a_n z^n$ and $g(z)=\sum_{n=1}^\infty b_n z^n$. If $\Lambda_f(z)\leq M$ for all $z\in\Bbb{D}$ and $|\lambda_f(0)|=1$, then $f$ is univalent in the disk 
$\Bbb{D}_{\rho_4}$ and $f(\Bbb{D}_{\rho_4})$ contains a schlicht disk $\Bbb{D}_{\sigma_4}$, where $\rho_4$ is the least positive root of the equation 
\bea\label{a0}&& 1-\frac{Mr}{1-r}\left(1+\frac{k}{\sqrt{k^2+1}}\right)=0\\[2mm]\text{and}
 &&\sigma_4=\rho_4+M(\rho_4+\log(1-\rho_4))\left(1+\frac{k}{\sqrt{k^2+1}}\right).\nonumber\eea
\end{theoE}
\begin{rem}
\textrm{Theorem E} of Wang and Zhong \cite{WZ2026} gives a univalence radius $\rho_4$, which is the least positive root of the equation (\ref{a0}). For each $\varepsilon>0$, we 
have $\Lambda_f(z)\leq M<M+\varepsilon$ for all $z\in\mathbb{D}$.
Applying \textrm{Lemma} \ref{lem14} with the bound $M+\varepsilon$, we find that
$f$ is univalent in $\mathbb{D}_{1/(M+\varepsilon)}$. Letting
$\varepsilon\to0^{+}$, these disks increase to $\mathbb{D}_{1/M}$.
Hence, $f$ is univalent in $\mathbb{D}_{1/M}$. The radius $1/M$ is larger than $\rho_4$. For $M>1$, the radius $1/M$ is sharp for the class
specified by the explicit hypotheses of \textrm{Theorem~E}.
Indeed, if we consider the holomorphic function
\beas
F_M(z)
=
M^2z-\int_0^z\frac{M^3-M}{M+\zeta}\,d\zeta,
\eeas
then $F_M(0)=0$ and
\beas
F_M'(z)
=
M\frac{z+1/M}{1+z/M}.
\eeas
Thus, we have $F_M'(0)=1$ and $|F_M'(z)|<M$ for
$z\in\mathbb{D}$. Therefore, $F_M(z)=h(z)+\ol{g(z)}$ with $g\equiv0$ satisfies all the hypotheses of \textrm{Theorem~E}.
Furthermore, $F_M$ has a critical point at $z=-1/M$ and consequently, $F_M$ cannot be univalent in any disk $\Bbb{D}_r$ with $r>1/M$. This proves the sharpness of the radius 
$1/M$.
\end{rem}
In 2024, Liu and Xu \cite{LX2024} obtained the following Landau-type theorem for $K$-quasiregular harmonic mapping with bounded analytic part.
\begin{theoF}\cite{LX2024}\label{1LiuXu2024}
Let $f = h + \ol{g}$ be a $K$-quasiregular harmonic mapping in $\Bbb{D}$ with $f(0)=h(0)=0$, $\lambda_f(0)=1$, and $|h(z)|\leq M$ for all $z\in\Bbb{D}$. Then $f$ is univalent on $\Bbb{D}_{\rho_5}$ and $f(\Bbb{D}_{\rho_5})$ contains a schlicht disk $\Bbb{D}_{\sigma_5}$, where
\beas
\rho_5= \frac{K+1}{8KM}\quad\text{and}\quad 
\sigma_5 = \frac{2KM}{K+1}\left\{1+\left(\left(\frac{4KM}{K+1}\right)^2-1\right)\log\!\left(1-\frac{(K+1)^2}{16K^2M^2}\right)\right\}.
\eeas
\end{theoF}
The Landau-type theorem for various classes of harmonic mappings has been studied in \cite{CG2011, CPW2011, G2006, H2008, 1L2009, 2L2009, LC2018, Z2015,ABMY2025, 1AK2024,BL2019,LL2021,LL2023}.
The above discussion naturally motivates the improvement of \textrm{Theorems E} and \textrm{F}, and the establishment of new Landau-type theorems for 
$K$-quasiregular harmonic mappings under hypotheses involving boundedness of $h$, $h'$, $\text{Re}(h)$, and the Bloch space condition on $h$. The present paper provides affirmative answers to all these objectives.

\medskip
\noindent\textbf{Novelty and significance.}
In this paper, we establish Landau-type theorems for $K$-quasiregular harmonic mappings $f=h+\overline{g}$ in the unit disk $\mathbb{D}$. The main novelty of this paper is the derivation of 
explicit, substantially improved univalence and schlicht disk radii under several natural boundedness hypotheses on the analytic part $h$, namely boundedness of
$\text{Re}(h)$, the Bloch space condition, boundedness of $h$, boundedness of $h'$, and boundedness of the minimum distortion $\lambda_f$. In particular, the 
paper obtains the explicit radii
\beas
  r_1=\frac{\pi(K+1)}{12\sqrt3\,KM},\quad r_5=\frac{K+1}{3\sqrt3\,KB}, \quad r_{9}= \frac{(K+1)\lambda_f(0)}{3\sqrt{3}\,M K},\quad r_{11}=\frac{K+1}{2KM},
  \quad r_{13}=\frac{1}{KM},
\eeas
which improve the earlier estimates of Wang and Zhong \cite{WZ2026} and Liu and Xu \cite{LX2024} by factors of approximately $1.54$ and $4$ in the relevant cases, and which 
are proved under matched hypotheses. The paper also introduces new Landau-type theorems for bounded real part (\textrm{Theorem}~\ref{T3} and \textrm{Corollary}~\ref{cor1})
and for the Bloch space setting (\textrm{Theorem}~\ref{T5}), with examples showing that these hypotheses are genuinely more general than several earlier assumptions. The function
$h(z)=\log((1+z)/(1-z))$ is unbounded, has unbounded derivative and has unbounded real part, yet satisfies the Bloch condition with $B=2$. Thus, \textrm{Theorem}~\ref{T5} 
applies where \textrm{Theorem}~\ref{T3}, \textrm{Proposition} \ref{T2} and \textrm{Corollary}~\ref{T1} do not. Numerical tables and figures illustrate the quantitative improvements and the relative sizes 
of the univalence and schlicht disks. Overall, these results advance the theory of Landau--Bloch theorems for quasiregular harmonic mappings and provide more effective
radius estimates for further applications.\\[2mm]
The remaining portion of the paper is organized as follows. In Section~\ref{sec2}, we present the three lemmas that are essential for
the proofs of the main results, the last of which is the sharp strip estimate for analytic functions with bounded real part. In Section~\ref{sec3}, we establish our main Landau-type 
theorems for $K$-quasiregular harmonic mappings. More precisely, \textrm{Theorem}~\ref{T3} and \textrm{Corollary}~\ref{cor1} deal with the case of bounded real part; \textrm{Theorem}
~\ref{T5} and \textrm{Corollary}~\ref{T1} deal with the Bloch space and the bounded analytic part, respectively; \textrm{Theorem}~\ref{T6} treats the normalization $J_f(0)=1$ with bounded $h$; 
\textrm{Proposition}~\ref{T2} treats bounded $h'$ and \textrm{Proposition}~\ref{T0} treats bounded minimum distortion $\lambda_f$.
\section{Some lemmas}\label{sec2}
To prove our main results, the following lemmas play a crucial role.
\begin{lem}\label{lem0}\cite{DP2008}  If $h(z)=\sum_{n=0}^\infty a_n z^n$ is analytic in $\Bbb{D}$ with $|h(z)|\leq1$ for $z\in\Bbb{D}$. Then, we have 
\beas \left|\frac{h^{(n)}(z)}{n!}\right|\leq \frac{1-|h(z)|^2}{(1-|z|)^{n-1}(1-|z|^2)}\quad\text{and}\quad |a_n|\leq 1-|a_0|^2\;\;\text{for each}\;\; n\geq 1\;\;\text{and}\;\; z\in\Bbb{D}.\eeas
\end{lem}
\begin{lem}\cite{H2014, LC2018}\label{lem14}
Let $f=h+\ol{g}$ be a harmonic mapping in the unit disk $\Bbb{D}$ such that $f(0)=0$ and $\lambda_f(0)=1$. 
\begin{enumerate}
\item[(i)] If $\Lambda_f(z)\leq 1$ for $z\in\Bbb{D}$, then $f$ is univalent in $\Bbb{D}$ and $f(\Bbb{D})$ contains a schlicht disk $\Bbb{D}$ and the result is sharp.
\item[(ii)] If $\Lambda_f(z)<M$ for $z\in\Bbb{D}$, then $M>1$ and $f$ is univalent in $\Bbb{D}_\rho$ with $\rho=1/M$, and $f(\Bbb{D}_\rho)$ contains a schlicht disk $\Bbb{D}_\sigma$, where 
\beas \sigma=M+(M^3-M)\log\left(1-\frac{1}{M^2}\right).\eeas
The result is sharp.
\end{enumerate}
\end{lem}
\begin{lem}\label{KVlem}\cite[Theorem 1.8]{KV2012}
Let $u$ be a real harmonic function of the unit disk $\Bbb{D}$ into $(-1,1)$. For $z=x+iy$, we write
$\nabla u(z):=(u_x(z),u_y(z))$ and $|\nabla u(z)|:=\sqrt{u_x(z)^2+u_y(z)^2}$.
Then the following sharp inequality holds:
\beas |\nabla u(z)|\leq \frac{4(1-|u(z)|^2)}{\pi(1-|z|^2)},\qquad z\in\mathbb{D}.\eeas
\end{lem}
\begin{lem}\label{lem2}
Let $M>0$ and let $h$ be analytic in $\mathbb{D}$ with $|\text{Re}(h(z))|<M$ for $z\in\mathbb{D}$.
Then, we have 
\beas
  |h'(z)|\le \frac{4M}{\pi\,(1-|z|^2)}\quad \text{for}\quad z\in\mathbb{D}.
\eeas
The constant $4M/\pi$ is sharp. Equality at $z=0$ is attained by
\beas
  h_0(z)=\frac{2M}{i\pi}\,\log\left(\frac{1+z}{1-z}\right), \quad z\in\Bbb{D},
\eeas
where $\log$ denotes the principal branch of the logarithm.
\end{lem}
\begin{proof}
Let $u(z)=\text{Re}(h(z))/M$. Then, $u$ is real harmonic in
$\mathbb{D}$ with $|u(z)|<1$ in $\Bbb{D}$. In view of \textrm{Lemma} \ref{KVlem}, we have
\beas
  |\nabla u(z)|\le \frac{4(1-|u(z)|^2)}{\pi(1-|z|^2)},\qquad z\in\mathbb{D}.
\eeas
Since $h/M$ is analytic with real part $u$, the Cauchy--Riemann
equations give
\beas
  \frac{|h'(z)|}{M}=\left|\left(\frac{h}{M}\right)'(z)\right|=|\nabla u(z)|.
\eeas
Hence
\beas
  |h'(z)|\le \frac{4M\left(1-|\text{Re}(h(z))|^2/M^2\right)}{\pi(1-|z|^2)}\leq \frac{4M}{\pi(1-|z|^2)},\qquad z\in\mathbb{D}.
\eeas
To prove the sharpness, we consider the following mapping
\beas
  h_0(z)=\frac{2M}{i\pi}\log\frac{1+z}{1-z},
\eeas
where $\log$ denotes the principal branch of the logarithm. Since the
M\"obius transformation $\zeta\mapsto(1+\zeta)/(1-\zeta)$ maps $\mathbb{D}$
onto the right half-plane, its principal logarithm has imaginary
part in $(-\pi/2,\pi/2)$. Consequently, $h_0$ is analytic in $\mathbb{D}$, $h_0(0)=0$ and
\beas
  \text{Re}(h_0(z))
  =\frac{2M}{\pi}\arg\left(\frac{1+z}{1-z}\right)\in(-M,M).
\eeas
Thus, $h_0$ satisfies the hypothesis. Moreover,
\beas
  h_0'(z)=\frac{4M}{i\pi(1-z^2)} \quad\text{and}\quad |h_0'(0)|=\frac{4M}{\pi}.
\eeas
Thus equality holds in the estimate at $z=0$, and the constant
$4M/\pi$ cannot be improved. This completes the proof.
\end{proof}
\section{Main results}\label{sec3}
In the following result, we establish a Landau-type theorem for a $K$-quasiregular harmonic mapping under the hypothesis that $\textrm{Re}(h(z))$ is bounded.
\begin{theo}\label{T3}
Let $f=h+\ol{g}$ be a $K$-quasiregular harmonic mapping in the unit disk $\Bbb{D}$ such that $f(0)=0$, $\lambda_f(0)=1$ and $|\textrm{Re}(h(z))| < M$ for $z\in\Bbb{D}$, 
where $M>0$. Then, $M\geq \pi(K+1)/(8K)$ and $f$ is univalent in the disk $\Bbb{D}_{r_1}$ and $f(\Bbb{D}_{r_1})$ contains a schlicht disk $\Bbb{D}_{r_2}$, where
\beas r_1=\frac{\pi (K+1)}{12\sqrt{3}\ KM}\quad\text{and}\quad r_2=\frac{4\sqrt{3}KM}{\pi(K+1)}\left( 1 + \left(\ \frac{144K^2M^2}{\pi^2(K+1)^2}- 1\right)\log\left(1-\frac{\pi^2(K+1)^2}{144K^2M^2}\right)\right).
\eeas
\end{theo}
\begin{proof}
Since $f$ is a $K$-quasiregular harmonic mapping in the unit disk $\Bbb{D}$, it follows that
\beas |g'(z)|\leq \frac{K-1}{K+1}|h'(z)|\quad\text{ for}\quad z\in\Bbb{D}.\eeas
In view of \textrm{Lemma} \ref{lem2}, we obtain 
\bea\label{r1} \Lambda_f(z)=|h'(z)|+|g'(z)|\leq \frac{2K}{K+1}|h'(z)|\leq \frac{8MK}{\pi(K+1)(1-|z|^2)}.\eea
Using (\ref{r1}) and $\lambda_f(0)=1$, we obtain
\beas 1=\lambda_f(0)\leq \Lambda_f(0)\leq \frac{8MK}{\pi(K+1)}.\eeas
Let $P(z)=C f(z/C)$ for $z\in\Bbb{D}$, where the parameter $C$ is a positive real number satisfying $C > 1$. Then, we have $\lambda_P(0)=\lambda_f(0)=1$ and
\beas \Lambda_P(z)=\left|P_z(z)\right|+\left|P_{\ol{z}}(z)\right|=\Lambda_f(z/C)\leq \frac{8MK}{\pi (K+1)\left(1-|z|^2/C^2\right)}.\eeas
Since $|z|<1$, we have $1-|z|^2/C^2>1-1/C^2$ and 
\beas
\Lambda_P(z)< \frac{8MK}{\pi (K+1)(1-1/C^2)} = \frac{8KMC^2}{\pi (K+1)(C^2-1)} := M_P(C). \eeas
In view of \textrm{Lemma \ref{lem14}}, we have the mapping $P(z)$ is univalent on $\Bbb{D}_{\rho_1(C)}$ with $\rho_1(C)= 1/M_P(C)$ and $P(\Bbb{D}_{\rho_1(C)})$ contains a schlicht disk of radius
\beas
\sigma_1(C)= M_P(C)+ \left(M_P^3(C)-M_P(C)\right)\log\left(1- \frac{1}{M_P^2(C)}\right).\eeas
Consequently, $f(z)=P(Cz)/C$ is univalent on $\Bbb{D}_{r_{1}(C)}$ and $f(\Bbb{D}_{r_{1}(C)})$ contains a schlicht disk of radius $r_{2}(C)=\sigma_1(C)/C$, where
\beas
r_{1}(C)= \frac{\rho_1(C)}{C} = \frac{1}{C M_P(C)} =\frac{\pi (K+1)(C^2-1)}{8KMC^3}. \eeas
Let 
\beas H(x)=\frac{x^2-1}{x^3}= \dfrac{1}{x}-\dfrac{1}{x^3}\quad \text{for}\quad x>1.\eeas
Our aim is to determine the maximum value of $H(x)$ for $x>1$.
Differentiating $H(x)$ with respect to $x$, we obtain 
\beas
H'(x)= -\frac{1}{x^2}+\frac{3}{x^4}= \frac{3-x^2}{x^4}.\eeas
It is easy to see that $H'(x)>0$ for $x\in(1,\sqrt{3})$ and $H'(x)<0$ for $x>\sqrt{3}$. Thus, $H(x)$ attains its maximum at $x=\sqrt{3}$. For $C=\sqrt{3}$, it is easy to see that
\beas M_P(\sqrt{3})= \frac{12KM}{\pi (K+1)}>1\quad \text{for}\quad K\geq 1 \quad\text{and}\quad M\geq \frac{\pi(K+1)}{8K}. \eeas
Thus the optimal univalence radius of $f$ is 
\beas
r_1=\max_{C>1} r_1(C)=r_1(\sqrt{3})=\frac{\pi (K+1)}{12\sqrt{3}\,KM}\eeas
and the image $f(\Bbb{D}_{r_1})$ contains a schlicht disk $\Bbb{D}_{r_2}$, where
\beas
r_2= \frac{4\sqrt{3}KM}{\pi(K+1)}\left( 1 + \left(\ \frac{144K^2M^2}{\pi^2(K+1)^2}- 1\right)\log\left(1-\frac{\pi^2(K+1)^2}{144K^2M^2}\right)\right).\eeas
This completes the proof.
\end{proof}
\noindent The numerical values of the univalence radii $r_1$ and the corresponding schlicht disk radii $r_2$ for various choices of $K$ and admissible $M$ are given in Table \ref{tab3}.
\begin{table}[H]
\centering
\begin{tabular}{|c|c|c|c|}
\hline
$K$ & $M$ & $r_1 $ & $r_2$ \\\hline
$1.0$ & $0.80$ & $0.3779$ & $0.2238$ \\\hline
$1.0$ & $1.00$ & $0.3023$ & $0.1672$ \\\hline
$1.0$ & $1.50$ & $0.2015$ & $0.1051$ \\\hline
$1.5$ & $0.70$ & $0.3599$ & $0.2092$ \\\hline
$1.5$ & $1.00$ & $0.2519$ & $0.1348$ \\\hline
$2.0$ & $0.60$ & $0.3779$ & $0.2238$ \\\hline 
$2.0$ & $0.80$ & $0.2834$ & $0.1547$ \\\hline
$2.0$ & $1.00$ & $0.2267$ & $0.1197$ \\\hline
$3.0$ & $0.55$ & $0.3664$ & $0.2144$ \\\hline
$3.0$ & $0.80$ & $0.2519$ & $0.1348$ \\\hline
\end{tabular}\caption{Univalence radii $r_1$ and schlicht disk radii $r_2$ for Theorem \ref{T3} for various values of $K$ and admissible $M$, where $M\geq \pi(K+1)/(8K)$}
\label{tab3}
\end{table}
A visual comparison of the univalence disk and the corresponding schlicht disk for $K=1.5$, $M=0.7$ are shown in Figure \ref{fig3}, which 
illustrates their relative sizes side by side.
\begin{figure}[H]
\centering
\includegraphics[scale=0.7]{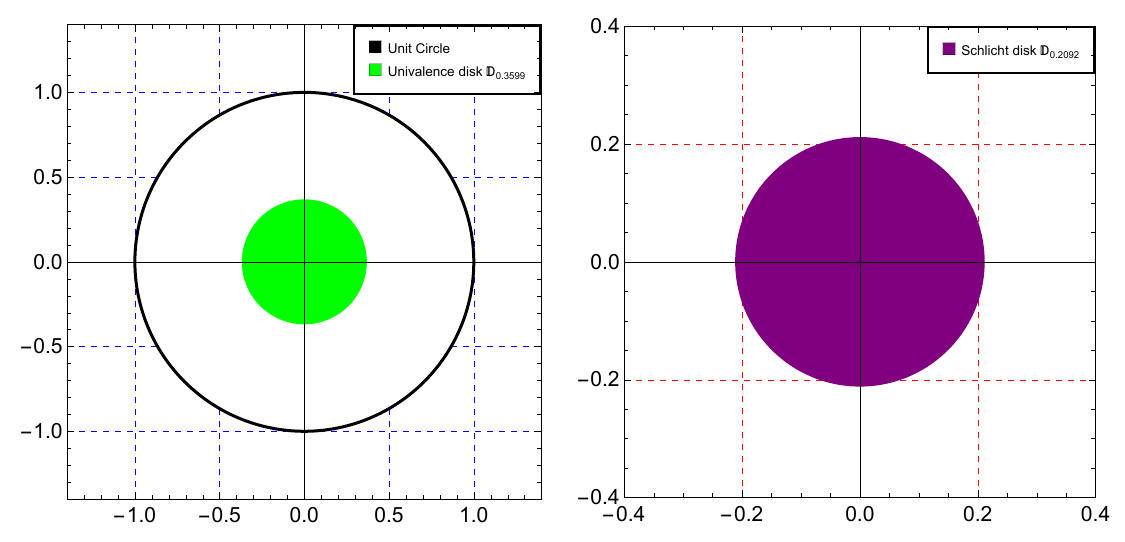}
\caption{Univalence disk $\Bbb{D}_{0.3599}$ and the corresponding schlicht disk $\Bbb{D}_{0.2092}$ for 
Theorem \ref{T3} with $K=1.5$ and $M=0.7$.}
\label{fig3}
\end{figure}
\noindent As a special case of Theorem \ref{T3}, we obtain the following result for analytic functions with bounded real part.
\begin{cor}\label{cor1}
Let $f$ be an analytic function in $\Bbb{D}$ with $f(0)=0$, $f'(0)=1$, and $|\text{Re}(f(z))| < M$ for $z\in\Bbb{D}$, where $M>0$. Then $M\geq \pi/4$ and $f$ is univalent in $\Bbb{D}_{r_{3}}$ and $f(\Bbb{D}_{r_{3}})$ contains a schlicht disk $\Bbb{D}_{r_{4}}$, where
\beas
r_{3} = \frac{\pi }{6\sqrt{3}\,M}\quad \text{and}\quad
r_{4} = \frac{2\sqrt{3}M}{\pi }\left( 1 + \left(\ \frac{36M^2}{\pi^2 }- 1\right)\log\left(1-\frac{\pi^2 }{36M^2}\right)\right).
\eeas
\end{cor}
\begin{proof}
The result follows immediately by setting $K=1$ in \textrm{Theorem} \ref{T3}.
\end{proof}
\begin{rem}
In the proof of \textrm{Theorem} \ref{T3}, one may directly define the scaled function $P(z)=\sqrt{3} f\left(z/\sqrt{3}\right)$ for $z\in\Bbb{D}$,
which gives $\Lambda_P(z) \leq 12KM/(\pi (K+1))$.
Applying \textrm{Lemma} \ref{lem14} directly to $P$ gives the same radii without the need for optimization. We have retained the general parameter $C$ in proof of 
the \textrm{Theorem} \ref{T3} to demonstrate that $C=\sqrt{3}$ is indeed optimal.
\end{rem}
\noindent We now illustrate \textrm{Theorem} \ref{T3} with the following example.
\begin{exam} We construct a harmonic mapping $f = h + \ol{g}$ that satisfies the conditions of \textrm{Theorem} \ref{T3} for $K =3/2$.
We know that $\phi(z)=(1/2)\log ((1+z)/(1-z))$ maps the unit disk $\Bbb{D}$ onto the horizontal strip $\{w\in\Bbb{C}: -\pi/4< Im(w)<\pi/4\}$.
Let
\beas h(z) =a\,i\,\log\frac{1+z}{1-z}\quad\text{for}\quad z\in\Bbb{D},\eeas
where the principal branch of the logarithm is taken so that $\log 1 = 0$.  
Then, $h$ is analytic in $\Bbb{D}$ with $h(0)=0$ and $h'(z) =2ai/\left(1-z^{2}\right)$. Let $g(z) = k\,h(z)$, where $k=(K-1)/(K+1)=1/5$. Then, we have
\beas \lambda_f(0) = |h'(0)| - |g'(0)| = |h'(0)| - k|h'(0)| = 4|h'(0)|/5.\eeas
By setting $\lambda_f(0)=1$ gives $|h'(0)| =5/4$. Furthermore, we have $|h'(0)|=2|a|$ and it follows that $|a|=5/8$. We choose $a=5/8$.
Since $\text{Re}(h(z)) = (5/4)\text{Re}\left(i\phi(z)\right)$, it follows that the mapping $h$ maps the unit disk conformally onto the vertical strip $\{w\in\Bbb{C} : |\text{Re}(w)| <5\pi/16\}$. Therefore,
\beas f(z) = h(z) + \ol{g(z)} = \frac{5}{8}\,i\,\log\frac{1+z}{1-z} - \frac{i}{8}\ \ol{\log\frac{1+z}{1-z}}.\eeas
is a $K$-quasiregular harmonic mapping in $\Bbb{D}$ with $f(0)=0$, $\lambda_f(0)=1$ and $|\text{Re}(h(z))|<M=5\pi/16$ for $z\in\Bbb{D}$.
Thus, $f$ satisfies all conditions of \textrm{Theorem} \ref{T3}. In view of \textrm{Theorem} \ref{T3}, we have that $f$ is univalent in the disk $|z| <r^*$, where
\beas
r^*= \frac{\pi(K+1)}{12\sqrt{3}\,K M} = \frac{5\pi /2}{12\sqrt{3}\cdot(3/2)\cdot(5\pi/16)}=\frac{4}{9 \sqrt{3}} \approx 0.2566.\eeas
\end{exam}
In the following result, we establish a Landau-type theorem for $K$-quasiregular harmonic mappings under the assumption that the analytic part belongs to the Bloch space. 
\begin{theo}\label{T5}
Let $f = h + \overline{g}$ be a $K$-quasiregular harmonic mapping in $\Bbb{D}$ such that $f(0) = 0$, $\lambda_f(0) = 1$. If $h$ belongs to the Bloch space, {\it i.e.},
$\sup_{z\in\Bbb{D}} (1-|z|^2)|h'(z)| \leq B$ for some $B> 0$, then $B\geq 1$, $f$ is univalent in the disk $\Bbb{D}_{r_{5}}$ and $f(\Bbb{D}_{r_{5}})$ contains a schlicht disk $\Bbb{D}_{r_{6}}$, where
\beas
r_{5}= \frac{K+1}{3\sqrt{3}\,K B}\quad \text{and}\quad r_{6}=\frac{\sqrt{3}K B}{K+1}\left(1 + \left(\frac{9K^2 B^2}{(K+1)^2} - 1\right)\log\left(1 - \frac{(K+1)^2}{9K^2 B^2}\right)\right).\eeas

\end{theo}
\begin{proof}
Since $f$ is a $K$-quasiregular harmonic mapping in the unit disk $\Bbb{D}$, we have
\beas |g'(z)|\leq \frac{K-1}{K+1}|h'(z)|\quad\text{and}\quad \Lambda_f(z)=|h'(z)|+|g'(z)|\leq \frac{2K}{K+1}|h'(z)|\quad \text{ for}\quad z\in\Bbb{D}.\eeas
From the Bloch condition, we have
\beas
|h'(z)| \leq \frac{B}{1-|z|^2} \quad \text{for all } z \in \Bbb{D}.\eeas
Thus, we have 
\beas
\Lambda_f(z) \leq \frac{2K}{K+1} \cdot \frac{B}{1-|z|^2} = \frac{L}{1-|z|^2},
\eeas
where $L =2K B/(K+1)$. Since $f$ is sense-preserving and $\lambda_f(0)=|h'(0)|-|g'(0)|=1$, it follows that 
\beas  B \geq |h'(0)| = 1 + |g'(0)| \geq 1.\eeas
Using a similar argument to those used in the proof of \textrm{Theorem} \ref{T3}, we have $f(z)$ is univalent in $\Bbb{D}_{r_{5}}$, where
\beas
r_{5} =\frac{K+1}{3\sqrt{3}K B}.
\eeas
The image $f(\Bbb{D}_{r_{5}})$ contains a schlicht disk of radius
\beas
r_{6}= \frac{\sqrt{3}K B}{K+1}\left(1 + \left(\frac{9K^2 B^2}{(K+1)^2} - 1\right)\log\left(1 - \frac{(K+1)^2}{9K^2 B^2}\right)\right).
\eeas
This completes the proof.
\end{proof}
\noindent As a special case of \textrm{Theorem} \ref{T5}, we obtain the following result for a $K$-quasiregular harmonic mapping with bounded analytic part.
\begin{cor}\label{T1}
Let $f = h+\ol{g}$ be a $K$-quasiregular harmonic mapping in the unit disk $\Bbb{D}$ such that $f(0)=0$ and $\lambda_f(0)=1$. If $|h(z)|\leq M$ for $z\in\Bbb{D}$,
then $M\geq 1$, $f$ is univalent in the disk $\Bbb{D}_{r_7}$ and $f(\Bbb{D}_{r_7})$ contains a schlicht disk $\Bbb{D}_{r_8}$, where
\beas
r_7= \frac{K+1}{3\sqrt{3} KM}\quad\text{and}\quad 
r_8 = \frac{\sqrt{3}\,KM}{K+1}\left(1+\left(\frac{9K^2M^2}{(K+1)^2}-1\right)\log\!\left(1-\frac{(K+1)^2}{9K^2M^2}\right)\right).
\eeas
\end{cor}
\begin{proof}
Since $h$ is analytic in $\Bbb{D}$ with $|h(z)|\leq M$ for $z\in\Bbb{D}$, in view of \textrm{Lemma} \ref{lem0}, we have
\bea\label{a2a}
|h'(z)|\leq \frac{M(1-|h(z)/M|^2)}{1-|z|^2}\leq  \frac{M}{1-|z|^2},\quad z\in\Bbb{D}.
\eea
Hence, $h$ belongs to the Bloch space with $B=M$. 
The result follows immediately by setting $B=M$ in \textrm{Theorem} \ref{T5}.
This completes the proof.
\end{proof}
\begin{rem}
The univalence radius obtained in \textrm{Theorem F} of Liu and Xu \cite{LX2024} 
is $\rho_5= (K+1)/(8KM)$, whereas \textrm{Corollary \ref{T1}} gives 
$r_7 = (K+1)/(3\sqrt{3}KM)$ without imposing any additional assumptions 
on the mapping. Since
\beas
\frac{r_7}{\rho_5} = \frac{8}{3\sqrt{3}} \approx 1.5396 > 1,
\eeas
our result improves the univalence radius by a factor of approximately $1.54$.
\end{rem}
\noindent As an application of \textrm{Theorem} \ref{T5}, we obtain the following example.
\begin{exam}\label{ex1}
Let 
\beas
h(z) = \log\left(\frac{1+z}{1-z}\right), \qquad z \in \Bbb{D},
\eeas
where the principal branch of the logarithm is chosen so that $\log 1 = 0$. Then $h$ is analytic in $\Bbb{D}$ with $h(0)=0$ and $h'(z) = 2/(1-z^2)$.
It is evident that
\beas
(1-|z|^2)|h'(z)|= (1-|z|^2)\left|\frac{2}{1-z^2}\right|\leq \frac{2(1-|z|^2)}{1-|z|^2} = 2.\eeas
Thus, $h$ belongs to the Bloch space with $B= 2$.
Let $g(z) = k h(z)$, where $k =(K-1)/(K+1)$ for some $K \geq 1$. Then, $f = h + \ol{g}$ is a $K$-quasiregular harmonic mapping in $\Bbb{D}$.
It is easy to see that $f(0)=0$ and $\lambda_f(0) = |h'(0)| - |g'(0)| = 2(1-k) =4/(K+1)$.
To ensure $\lambda_f(0)=1$, we take $K = 3$, which implies $k = 1/2$. Therefore,
\beas
f(z) = \log\left(\frac{1+z}{1-z}\right) + \frac{1}{2}\overline{\log\left(\frac{1+z}{1-z}\right)}.
\eeas
is a $3$-quasiregular harmonic mapping with $f(0)=0$, $\lambda_f(0)=1$ and $h$ is in the Bloch space with $B=2$.
Using \textrm{Theorem} \ref{T5} with $K=3$ and $B=2$, we have that $f$ is univalent in $\Bbb{D}_{r_{5}^*}$, where
\beas
r_{5}^*= \frac{K+1}{3\sqrt{3}\,K B} = \frac{2}{9\sqrt{3}} \approx 0.1283.
\eeas
The image $f(\Bbb{D}_{r_{5}^*})$ contains a schlicht disk $\Bbb{D}_{r_{6}^*}$, where
\beas
r_{6}^* &=& \frac{\sqrt{3}K B}{K+1}\left(1 + \left(\frac{9K^2 B^2}{(K+1)^2} - 1\right)\log\left(1 - \frac{(K+1)^2}{9K^2 B^2}\right)\right) \\
&=& \frac{3\sqrt{3}}{2}\left(1 + \frac{77}{4}\log\left(\frac{77}{81}\right)\right) \\
&\approx&0.0652.
\eeas
\end{exam}
\noindent The numerical values of the univalence radii $r_5$ and the corresponding schlicht disk radii $r_6$ for various choices of $K$ and $B$ are given in Table \ref{tab1}. The rows with $B=M$ coincide with the radii of the bounded case $|h(z)|\leq M$, while the rows with $B>M$ correspond to genuinely Bloch functions that are not covered by the bounded case (see \textrm{Remark} \ref{rem3}).
\begin{table}[H]
\centering
\begin{tabular}{|c|c|c|c|}
\hline
$K$ & $B$ & $r_5$ & $r_6$ \\
\hline
1 & 1 & 0.3849 & 0.2298 \\
\hline
1.1 & 2 & 0.1837 & 0.0951\\
\hline
1.5 & 1 & 0.3208 & 0.1800 \\
\hline
2 & 1 & 0.2887 & 0.1581 \\
\hline
2 & 2 & 0.1443 & 0.0737 \\
\hline
2 & 2.5 &0.1155 & 0.0585 \\
\hline
3 & 1 & 0.2566 & 0.1377 \\
\hline
3 & 1.5 & 0.1711& 0.0882 \\
\hline
\end{tabular}
\caption{Univalence radii $r_5$ and schlicht disk radii $r_6$ for \textrm{Theorem} \ref{T5} for various values of $K$ and $B$}
\label{tab1}
\end{table}
\noindent A visual comparison of the univalence disk and the corresponding schlicht disk for $K=1.5$, $B=1$ are shown in Figure \ref{fig1}, which 
illustrates their relative sizes side by side.
\begin{figure}[H]
\centering
\includegraphics[scale=0.7]{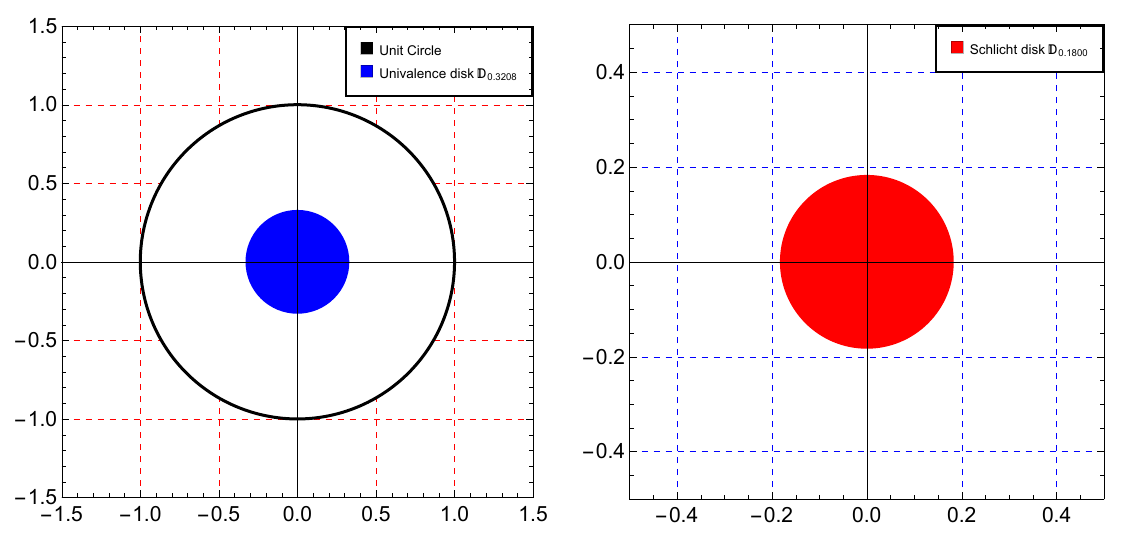}
\caption{Univalence disk $\Bbb{D}_{0.3208}$ and the corresponding schlicht disk $\Bbb{D}_{0.1800}$ for \textrm{Theorem} \ref{T5} with $K=1.5$ and $B=1$.}
\label{fig1}
\end{figure}
\noindent In the following result, we establish a Landau-type theorem for 
$K$-quasiregular harmonic mappings with bounded analytic part 
under the normalization $J_f(0)=1$.
\begin{theo}\label{T6}
Let $f = h+\ol{g}$ be a $K$-quasiregular harmonic mapping in the unit disk $\Bbb{D}$ such that $f(0)=0$ and $J_f(0)=1$. If $|h(z)|\leq M$ for $z\in\Bbb{D}$ and $M>0$, then 
$M\geq ((\lambda_f(0))^2+1)/(2\lambda_f(0))$ and $f$ is univalent in the disk $\Bbb{D}_{r_{9}}$ and $f(\Bbb{D}_{r_{9}})$ contains a schlicht disk $\Bbb{D}_{r_{10}}$, where
\bs\beas
r_{9}= \frac{(K+1)\lambda_f(0)}{3\sqrt{3}\,M K}\;\text{and}\;r_{10} =\frac{\sqrt{3}MK}{K+1}\left( 1 + \left(\frac{9M^2K^2}{(\lambda_f(0)(K+1))^2}- 1\right)\log\left(1-\frac{(\lambda_f(0)(K+1))^2}{9M^2K^2}\right)\right).\eeas\es
\end{theo}
\begin{proof}
Since $f$ is a $K$-quasiregular harmonic mapping in the unit disk $\Bbb{D}$, it follows that $\Lambda_f(z)\leq K\lambda_f(z)$ for $z\in\Bbb{D}$. Thus, we have
\bea\label{Na1} |g'(z)|\leq \frac{K-1}{K+1}|h'(z)|\quad\text{ for}\quad z\in\Bbb{D}.\eea
Since $J_f(0)=1$ and $\Lambda_f(z)\leq K\lambda_f(z)$ for $z\in\Bbb{D}$, it follows that 
\beas \lambda_f(0)=\frac{1}{\Lambda_f(0)}\geq \frac{1}{K\lambda_f(0)},\quad \textit{i.e.,}\quad \lambda_f(0)\geq \frac{1}{\sqrt{K}}.\eeas
Since $f$ is sense-preserving with $J_f(0)=1$, we have $\lambda_{f}(0)=|f_z(0)|-|f_{\ol{z}}(0)|>0$ and $(\Lambda_f(0))^2\geq \Lambda_f(0)\lambda_f(0)=J_f(0)=1$. 
Therefore, we have $\lambda_{f}(0)=1/\Lambda_f(0)\leq 1$.
Using (\ref{Na1}), we deduce that
\bea\label{Na1a} \Lambda_f(z)=|h'(z)|+|g'(z)|\leq \frac{2K}{K+1}|h'(z)|.\eea
Since $h$ is analytic in $\Bbb{D}$ with $|h(z)|\leq M$ for $z\in\Bbb{D}$, in view of \textrm{Lemma} \ref{lem0}, we obtain (\ref{a2a}). 
Therefore, we have
\bea\label{Na2} \Lambda_f(z)\leq \frac{2KM}{(K+1)(1-|z|^2)},\quad z\in\Bbb{D}. \eea
Using (\ref{a2a}) and  $J_f(0)=1$, we obtain
\beas M\geq |h'(0)|=\frac{\lambda_f(0)+\Lambda_f(0)}{2}=\frac{\lambda_f(0)+1/\lambda_f(0)}{2}\geq 1.\eeas
Consider the function $P(z)=f(z)/\lambda_f(0)$ for $z\in\Bbb{D}$. Then, we have $P(0)=0$, $\lambda_P(0)=1$ and
\beas \Lambda_P(z)=\frac{\Lambda_f(z)}{\lambda_f(0)}\leq \frac{2MK}{\lambda_f(0)(K+1)\left(1-|z|^2\right)}.\eeas
Let $Q(z)=C P(z/C)=C f(z/C)/\lambda_f(0)$ for $z\in\Bbb{D}$, where the parameter $C$ is a positive real number satisfying $C > 1$. Then, we have $\lambda_Q(0)=\lambda_P(0)=1$ and
\beas \Lambda_Q(z)=\left|Q_z(z)\right|+\left|Q_{\ol{z}}(z)\right|=\Lambda_P(z/C)\leq \frac{2MC^2K}{\lambda_f(0)(K+1)\left(C^2-|z|^2\right)}.\eeas
Since $|z|<1$, we have $C^2-|z|^2>C^2-1$ and 
\beas
\Lambda_Q(z)< \frac{2MC^2K}{\lambda_f(0)(K+1)\left(C^2-1\right)}:=M_Q(C). \eeas
In view of \textrm{Lemma \ref{lem14}}, we have the mapping $Q(z)$ is univalent on $\Bbb{D}_{\rho_2(C)}$ with $\rho_2(C)= 1/M_Q(C)$ and $Q(\Bbb{D}_{\rho_2(C)})$ contains a schlicht disk $\Bbb{D}_{\sigma_2(C)}$, where
\beas
\sigma_2(C)=M_Q(C)+(M_Q^3(C)-M_Q(C))\log\!\left(1-\frac{1}{M_Q^2(C)}\right).\eeas
Consequently, $f(z)=\lambda_f(0) Q(Cz)/C$ is univalent on $\Bbb{D}_{r_{9}(C)}$ and $f(\Bbb{D}_{r_{9}(C)})$ contains a schlicht disk of radius $r_{10}(C)=\lambda_f(0)\sigma_2(C)/C$, where
\beas
r_{9}(C)= \frac{\rho_2(C)}{C} = \frac{1}{C M_Q(C)} =\frac{\lambda_f(0)(K+1)\left(C^2-1\right)}{2MC^3K}. \eeas
For $x > 1$, it is evident that $(x^2 - 1)/x^3$ attains its maximum at $x = \sqrt{3}$. For $C=\sqrt{3}$, we have 
\beas M_Q(\sqrt{3})=\frac{3MK}{\lambda_f(0)(K+1)}>1.\eeas
Thus, the optimal univalence radius of $f$ is
\beas
r_{9}=\max_{C>1} r_{9}(C)=r_{9}(\sqrt{3})=\frac{(K+1)\lambda_f(0)}{3\sqrt{3}MK}\geq \frac{K+1}{3\sqrt{3}MK^{3/2}}\eeas
and the image $f(\Bbb{D}_{r_{9}})$ contains a schlicht disk $\Bbb{D}_{r_{10}}$, where
\beas r_{10}=\frac{\lambda_f(0)\sigma_2(\sqrt{3})}{\sqrt{3}}= \frac{\sqrt{3}MK}{K+1}\left( 1 + \left(\frac{9M^2K^2}{(\lambda_f(0)(K+1))^2}- 1\right)\log\left(1-\frac{(\lambda_f(0)(K+1))^2}{9M^2K^2}\right)\right).\eeas
This completes the proof.
\end{proof}
\noindent The numerical values of the univalence radii $r_{9}$ and the corresponding schlicht disk radii $r_{10}$ for various choices of $K$ and $M$ with $\lambda_f(0)=1/\sqrt{2}$ are given in Table \ref{tab5}. Note that, here $M\geq |h'(0)|=(\lambda_f(0)+1/\lambda_f(0))/2=3/(2\sqrt{2})\approx 1.061$.
\begin{table}[H]
\centering
\begin{tabular}{|c|c|c|c|}
\hline
$K$ & $M$ & $r_{9}$ & $r_{10}$ \\
\hline
2 & 1.07& 0.1908&0.0701 \\\hline
2&1.1&0.1856&0.0680\\\hline
2  & 1.2 & 0.1701&0.0620\\\hline
2.5 & 1.1 &0.1732&0.0632\\\hline 
2.5& 1.2 &0.1588&0.0576\\\hline
3  & 1.1 &0.1649&0.0600\\\hline
\end{tabular}
\caption{Univalence radii $r_{9}$ and schlicht disk radii $r_{10}$ for Theorem \ref{T6} for various values of $K$ and $M$ with $\lambda_f(0)=1/\sqrt{2}$.}
\label{tab5}
\end{table}
\noindent A comparison of the univalence radii obtained in \textrm{Corollary} \ref{T1} and \textrm{Theorem} \ref{T6} for $K=2$ is given in Table \ref{tab6}.
\begin{table}[H]
\centering
\begin{tabular}{|c|c|c|c|c|c|}
\hline
$M$ & $r_7$ &$r_9^{*}=r_{9}$ with $\lambda_f(0)=1/\sqrt{2}$& Ratio $r_7/r_{9}^{*}$ \\\hline
1.2 & 0.2406 & 0.1701&1.4142\\\hline
1.5 & 0.1925 & 0.1361&1.4142\\\hline
2.0 & 0.1443 & 0.1021&1.4142\\\hline
2.5 & 0.1155 & 0.0816&1.4142\\\hline
\end{tabular}\caption{Comparison of univalence radii in \textrm{Corollary} \ref{T1} and \textrm{Theorem} \ref{T6} for $K=2$.}
\label{tab6}
\end{table}
A visual comparison of the univalence disk and the corresponding schlicht disk for $K=2$, $M=1.07$ and $\lambda_f(0)=1/\sqrt{2}$ are shown in Figure \ref{fig5}, which 
illustrates their relative sizes side by side.
\begin{figure}[H]
\centering
\includegraphics[scale=0.75]{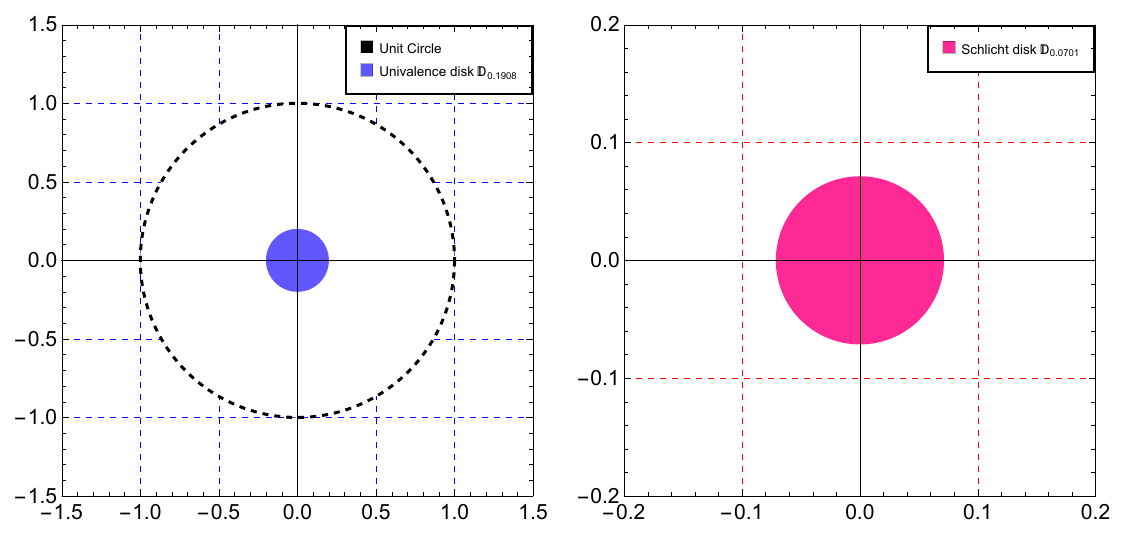}
\caption{Univalence disk $\Bbb{D}_{0.1908}$ and the corresponding schlicht disk $\Bbb{D}_{0.0701}$ for 
\textrm{Theorem} \ref{T6} with $K=2$, $M=1.07$ and $\lambda_f(0)=1/\sqrt{2}$.}
\label{fig5}
\end{figure}
\noindent In the following result, we establish a Landau-type theorem for a $K$-quasiregular harmonic mapping when the derivative of the analytic part is bounded.
\begin{prop}\label{T2}
Let $f = h + \overline{g}$ be a $K$-quasiregular harmonic mapping on the unit disk $\Bbb{D}$ such that $f(0)=0$ and $\lambda_f(0)=1$. If $|h'(z)|\leq M$ for $z\in\Bbb{D}$, 
then, $M\geq 1$ and the following hold:
\begin{enumerate}
\item[(i)] If $M > 1$ (or if $M=1$ with $K>1$), then $f$ is univalent in the disk $\Bbb{D}_{r_{11}}$ and $f(\Bbb{D}_{r_{11}})$ contains a schlicht disk $\Bbb{D}_{r_{12}}$, where
\beas r_{11}=\frac{K+1}{2KM}\quad\text{and}\quad r_{12}= \frac{2KM}{K+1} +\left(\left(\frac{2KM}{K+1} \right)^3 - \frac{2KM}{K+1}\right)\log\!\left(1 - \left(\frac{K+1}{2KM} \right)^2\right).\eeas
\item[(ii)] If $M=1$ and $K=1$, then $f$ is univalent in the entire unit disk $\Bbb{D}$ and $f(\Bbb{D})$ contains a schlicht disk $\Bbb{D}$.\end{enumerate}
\end{prop}
\begin{proof}
Since $f$ is a $K$-quasiregular harmonic mapping in the unit disk $\Bbb{D}$, we have $\Lambda_f(z)\leq K\lambda_f(z)$ for $z\in\Bbb{D}$. Thus, we have
\beas |g'(z)|\leq \frac{K-1}{K+1}|h'(z)|\quad\text{ for}\quad z\in\Bbb{D}.\eeas
Using $|h'(z)|\leq M$ for $z\in\Bbb{D}$, we obtain 
\beas \Lambda_f(z)=|h'(z)|+|g'(z)|\leq M + \frac{K-1}{K+1}M = \frac{2K}{K+1}M \quad \text{for}\quad z\in\Bbb{D}.\eeas
Since $f$ is sense-preserving and $\lambda_f(0)=|h'(0)|-|g'(0)|=1$, it follows that $M \geq |h'(0)| = 1 + |g'(0)| \geq 1$. 
We now consider two cases.\\[2mm]
\textbf{Case 1.} Let $M > 1$ (or if $M=1$ with $K>1$). Then, we have $L:=2KM/(K+1)>1$. For each $\varepsilon>0$,
$\Lambda_f(z)\leq L<L+\varepsilon$ for $z\in\mathbb{D}$.
In view of \textrm{Lemma \ref{lem14}}, we have $f$ is univalent in $\mathbb{D}_{1/(L+\varepsilon)}$
and the image $f(\mathbb{D}_{1/(L+\varepsilon)})$ contains a schlicht disk of radius 
\beas \sigma_{\varepsilon}=(L+\varepsilon)+\left((L+\varepsilon)^3-(L+\varepsilon)\right)\log\left(1-\frac{1}{(L+\varepsilon)^2}\right).\eeas
Letting $\varepsilon\to 0^{+}$, we conclude that $f$ is univalent
in $\mathbb{D}_{r_{11}}$, where $r_{11}=1/L$ and the image $f(\mathbb{D}_{1/L})$ contains a schlicht
disk of radius 
\beas
r_{12}= \frac{1}{r_{11}} + \left(\frac{1}{r_{11}^3}-\frac{1}{r_{11}}\right)\log\left(1-r_{11}^2\right).
\eeas
\textbf{Case 2.} Let $M=K=1$, then we have $\Lambda_f(z)\leq 1$ for $z\in\Bbb{D}$. By \textrm{Lemma} \ref{lem14}, we have $f$ is univalent in the unit disk $\Bbb{D}$ and $f(\Bbb{D})$ contains a schlicht disk $\Bbb{D}$. This completes the proof.
\end{proof}
\begin{rem}
The conclusions of \textrm{Proposition}~\ref{T2} also hold for harmonic
mappings $f=h+\ol{g}$ in $\Bbb{D}$ satisfying
\beas
|g'(z)|\leq k|h'(z)|,\qquad
f(0)=0,\quad \lambda_f(0)=1\quad\text{and}\quad 
|h'(z)|\leq M,
\eeas
where $k\in[0,1)$ and $K=(1+k)/(1-k)$.
Furthermore,
\beas
\Lambda_f(z)=|h'(z)|+|g'(z)|\leq (1+k)|h'(z)|\leq \frac{2KM}{K+1}
\eeas
For this broader class, the univalence radius
$r_{11}=(K+1)/(2KM)$ is sharp when $K=1$ for the following holomorphic function
\beas
f_0(z) = \int_0^z M\,\frac{1-M \eta}{M-\eta}\,d\eta,\quad\text{where}\quad M>1.\eeas
It is evident that $f_0$ satisfies $f_0(0)=0$ and $f_0'(0)=1$. Since the M\"{o}bius transformation $(1-Mz)/(M-z)$ has modulus $1$ on $|z|=1$ and is analytic in the unit disk $\Bbb{D}$, by the maximum modulus principle, we have $|f_0'(z)|\leq M$ for $z\in\Bbb{D}$.
Furthermore, $f_0$ has a critical point at $z=1/M<1$ and consequently, $f_0$ cannot be univalent in any disk of radius larger than $1/M$. 
\end{rem}
\begin{rem}
Under the normalization $h(0)=0$ and $|h'(z)|\leq M$ for $z\in\Bbb{D}$, we have
\beas
|h(z)|=\left|z\int_0^1 h'(tz)\,dt\right|\leq |z|\int_0^1 |h'(tz)|\,dt\leq M|z|<M,\quad z\in\Bbb{D}.
\eeas
Thus, \textrm{Theorem F} of Liu and Xu \cite{LX2024} also applies and guarantees univalence in the disk of radius $\rho_5=(K+1)/(8KM)$.
Under the stronger derivative hypothesis
$|h'(z)|\leq M$, \textrm{Proposition}~\ref{T2} gives $ r_{11}/\rho_5=4$.
 \end{rem}
\noindent The numerical values of the univalence radii $r_{11}$ and the corresponding schlicht disk radii $r_{12}$ for various choices of $K$ and $M$ are given in Table \ref{tab2}.
\begin{table}[H]
\centering
\begin{tabular}{|c|c|c|c|}
\hline
$K$ & $M$ & $r_{11}$ & $r_{12}$ \\
\hline
1&1.1&0.9091&0.6955\\\hline
1&1.5&0.6667&0.3979 \\\hline
1 & 2 & 0.5000 & 0.2739 \\\hline
1.1 & 1.01 & 0.9451&0.7751 \\\hline
2&1.1&0.6818&0.4110\\\hline
2 & 2 & 0.3750 & 0.1925 \\\hline
3 & 1.5 & 0.4444& 0.2385 \\\hline
\end{tabular}
\caption{Univalence radii $r_{11}$ and schlicht disk radii $r_{12}$ for \textrm{Proposition} \ref{T2} for various values of $K$ and $M$, where $M>1$}
\label{tab2}
\end{table}
A visual comparison of the univalence disk and the corresponding schlicht disk for $K=1.1$, $M=1.01$ are shown in Figure \ref{fig2}, which 
illustrates their relative sizes side by side.
\begin{figure}[H]
\centering
\includegraphics[scale=0.75]{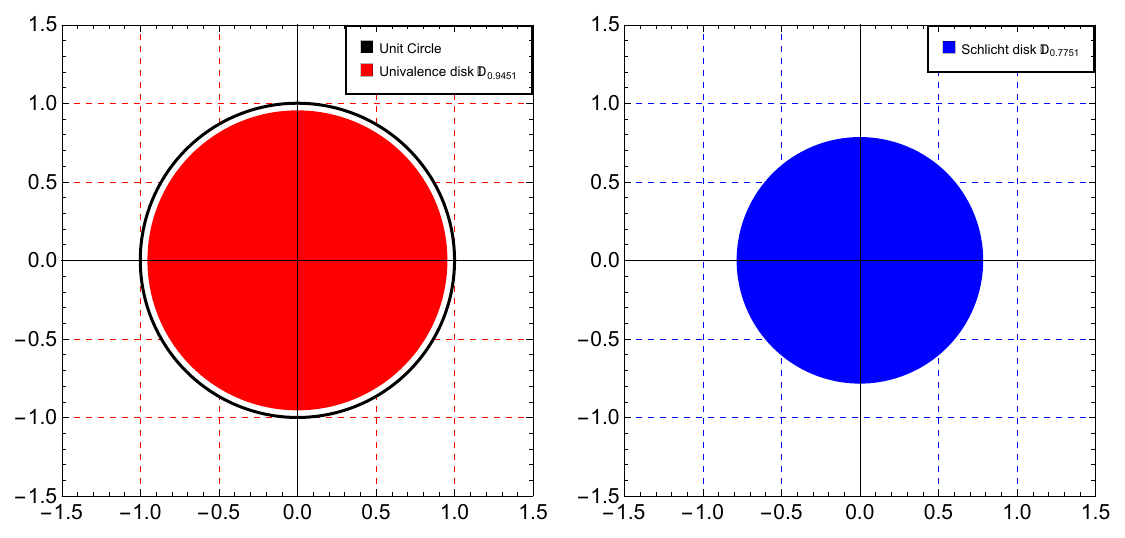}
\caption{Univalence disk $\Bbb{D}_{0.9451}$ and the corresponding schlicht disk $\Bbb{D}_{0.7751}$ for 
\textrm{Proposition} \ref{T2} with $K=1.1$ and $M=1.01$.}
\label{fig2}
\end{figure}
\begin{rem} Table \ref{tab2} illustrates that the values of the univalence radius $r_{11}$ in \textrm{Proposition} \ref{T2} can be made arbitrarily close to $1$. For a fixed  $K>1$ and letting $M\to 1^+$, we have 
\beas r_{11} \longrightarrow \frac{K+1}{2K},\eeas
which approaches $1$ as $K\to 1^+$. For example, if we choose $K=1.1$, the limiting value of $r_{11}$ is approximately $0.9545$, 
covering over $91.1\%$ of the unit disk. However, the corresponding univalence radius obtained in 
\cite{LX2024} is $(K+1)/(8KM)$, which is uniformly smaller by a factor of exactly $4$. This clearly illustrates the significant advantage of \textrm{Proposition} \ref{T2}.
\end{rem}
\begin{rem}\label{rem3}
The \textrm{Example} \ref{ex1} shows that the Bloch space condition is genuinely more general than the hypotheses of \textrm{Theorem} \ref{T3}, \textrm{Proposition} \ref{T2} 
and \textrm{Corollary} \ref{T1}. Indeed, for the function
\beas
h(z)=\log\left(\frac{1+z}{1-z}\right),
\eeas
we have the following:
\begin{itemize}
\item[(i)] $\text{Re}(h(z))$ is unbounded above in $\Bbb{D}$, so \textrm{Theorem} \ref{T3} does not apply;
\item[(ii)] $h(z)$ is unbounded in $\Bbb{D}$, so \textrm{Corollary} \ref{T1}  does not apply;
\item[(iii)] $h'(z)$ is unbounded in $\Bbb{D}$, so \textrm{Proposition} \ref{T2} does not apply.
\end{itemize}
This example shows that the Bloch space condition is strictly weaker than the hypotheses of \textrm{Theorem} \ref{T3}, \textrm{Proposition} \ref{T2} 
and \textrm{Corollary} \ref{T1}. Indeed, although $h$ is 
unbounded and fails the assumptions of those theorems, it satisfies the Bloch space condition with $B=2$, allowing \textrm{Theorem} \ref{T5} to give explicit radii. This 
demonstrates the advantage of considering the Bloch space setting.
\end{rem}
\noindent In the following result, we obtain a Landau-type theorem for a $K$-quasiregular harmonic mapping. 
\begin{prop}\label{T0}
Let $f=h+\ol{g}$ be a $K$-quasiregular harmonic mapping in the unit disk $\Bbb{D}$ such that $f(0)=0$ and $\lambda_f(0)=1$. If $\lambda_f(z)\leq M$ for all $z\in\Bbb{D}$, then the following hold:
\begin{enumerate}
\item[(i)] If $KM>1$, then $f$ is univalent in the disk $\Bbb{D}_{r_{13}}$ and $f(\Bbb{D}_{r_{13}})$ contains a schlicht disk $\Bbb{D}_{r_{14}}$, where
\beas
r_{13}=\frac{1}{KM}\quad\text{and}\quad r_{14}=KM+\left((KM)^3-KM\right)\log\left(1-\frac{1}{(KM)^2}\right).
\eeas
\item[(ii)] If $KM=1$, then $f$ is univalent in the unit disk $\Bbb{D}$ and $f(\Bbb{D})$ contains a schlicht disk $\Bbb{D}$. In fact, $f(z)=e^{i\theta}z$ for some $\theta\in\Bbb{R}$.
\end{enumerate}
\end{prop}
\begin{proof}
Since $f$ is a $K$-quasiregular harmonic mapping in the unit disk $\Bbb{D}$, it follows that $\Lambda_f(z)\leq K\lambda_f(z)$ for $z\in\Bbb{D}$. 
Using $\lambda_f(z)\leq M$ for all $z\in\Bbb{D}$, we obtain $\lambda_f(0)=1\leq M$ and $\Lambda_f(z) \leq KM$ for all $z\in\Bbb{D}$. We now consider two cases.\\[2mm]
\textbf{Case 1:} When $KM = 1$. Since $K \geq 1$ and $M \geq 1$, we must have $K = M = 1$. Then $\Lambda_f(z) = \lambda_f(z)$ for all $z \in \Bbb{D}$, which implies
$|g'(z)| = 0$ for all $z \in \Bbb{D}$. Since $g(0) = 0$, we have $g \equiv 0$ in $\Bbb{D}$. Thus, $f$ is an analytic function in $\Bbb{D}$ with $|f'(0)| = 1$ and $|f'(z)| \leq 1$ 
for all $z\in\Bbb{D}$. Applying the Maximum Modulus Principle to the analytic function $f'$, we obtain that $f'$ is constant in $\Bbb{D}$, and consequently $f'(z) \equiv e^{i\theta}$ for 
some $\theta \in \Bbb{R}$. Since $f(0)=0$, we have $f(z) = e^{i\theta}z$. Thus, $f$ is univalent in the unit disk 
$\Bbb{D}$ and $f(\Bbb{D}) = \Bbb{D}$ contains a schlicht disk of radius $1$.\\
\textbf{Case 2:} When $KM > 1$. For each $\varepsilon>0$,
$\Lambda_f(z)\le KM<KM+\varepsilon$ for all $z\in\mathbb{D}$.
Applying \textrm{Lemma} \ref{lem14} with the bound $KM+\varepsilon$ and then
letting $\varepsilon\to0^+$, we conclude that $f$ is univalent in $\Bbb{D}_{r_{13}}$ with $r_{13}=1/(KM)$ and
 $f(\Bbb{D}_{r_{13}})$ contains a schlicht disk $\Bbb{D}_{r_{14}}$, where
\beas
r_{14}=KM+\left((KM)^3-KM\right)\log\left(1-\frac{1}{(KM)^2}\right).
\eeas
This completes the proof.
\end{proof}
\begin{rem}
The conclusions of \textrm{Proposition}~\ref{T0} also hold for harmonic
mappings $f=h+\ol{g}$ in $\Bbb{D}$ satisfying
\beas
|g'(z)|\leq k|h'(z)|,\qquad
f(0)=0\quad \lambda_f(0)=1\quad\text{and}\quad \lambda_f(z)\leq M,
\eeas
where $k\in[0,1)$ and $K=(1+k)/(1-k)$. Indeed, $|g'(z)|\leq k|h'(z)|$ implies $\Lambda_f(z)\leq K\lambda_f(z)\leq KM$.
Unlike the sense-preserving class considered in \textrm{Proposition}~\ref{T0}, this broader class permits critical points. For this broader class, the univalence radius
$r_0=1/(KM)$ is sharp when $K=1$ for the following extremal function
\beas f_0(z)=\int_0^z M\,\frac{1-M\zeta}{M-\zeta}\,d\zeta,\quad \text{where}\quad M>1.\eeas
It is evident that $f_0$ 
has a critical point at $z=1/M<1$ and consequently, $f_0$ cannot be univalent in any disk $\Bbb{D}_r$ with $r>1/M$. 
Thus, $r_{13}=1/M$ is sharp for this broader class
when $K=1$.
\end{rem}
\section*{Declarations}
\noindent
{\bf Conflict of Interest:} The authors declare that there are no conflicts of interest regarding the publication of this paper.\\[1mm]
{\bf Availability of data and materials:} Not applicable.\\[1mm]
{\bf Funding:} This research did not receive any specific grant from funding agencies in the public, commercial, or not-for-profit sectors.\\[1mm]
{\bf Authors' contributions:} All authors contributed equally to the investigation of the problem, and all authors have read and approved the final manuscript.

\end{document}